\documentclass[11pt]{amsart}
\usepackage{amssymb}
\usepackage{amsthm,amsfonts}

\usepackage{amsfonts}
\usepackage{amsmath}
\usepackage[abbrev]{amsrefs}
\usepackage{enumerate}
\usepackage{color}

\newtheorem{Theorem}{Theorem}[section]

\newtheorem{Lemma}[Theorem]{Lemma}
\newtheorem{Corollary}[Theorem]{Corollary}

\newtheorem{Definition}[Theorem]{Definition}

\newtheorem{corollary}[Theorem]{Corollary}
\newtheorem{proposition}[Theorem]{Proposition}
\newtheorem{definition}[Theorem]{Definition}

\theoremstyle{definition}
\newtheorem{Remark}[Theorem]{Remark}

\begin{document}

\title{Diameter two properties in some vector-valued function spaces}

\author{Han Ju Lee}
\address{Department of Mathematics Education, Dongguk University - Seoul, 04620 (Seoul), Republic of Korea}
\email{hanjulee@dgu.ac.kr}

\author{Hyung-Joon Tag}
\address{Department of Mathematics Education, Dongguk University - Seoul, 04620 (Seoul), Republic of Korea}
\email{htag@dongguk.edu}

\date{2021. February 20.}

\thanks{
The first author was supported by Basic Science Research Program through the National Research Foundation of Korea(NRF) funded by the Ministry of Education, Science and Technology [NRF-2020R1A2C1A01010377].
The second author was supported by Basic Science Research Program through the National Research Foundation of Korea(NRF) funded by the Ministry of Education, Science and Technology [NRF-2020R1A2C1A01010377].
}

\begin{abstract}
We introduce a vector-valued version of a uniform algebra, called the vector-valued function space over a uniform algebra. The diameter two properties of the vector-valued function space over a uniform algebra on an infinite compact Hausdorff space are investigated. Every nonempty relatively weakly open subset of the unit ball of a vector-valued function space $A(K, (X, \tau))$ over an infinite dimensional uniform algebra has the diameter two, where $\tau$ is a locally convex Hausdorff topology on a Banach space $X$ compatible to a dual pair.  Under the assumption on $X$ being uniformly convex with norm topology $\tau$ and the additional condition that $A\otimes X\subset A(K, X)$, it is shown that Daugavet points and $\Delta$-points on $A(K, X)$ over a uniform algebra $A$ are the same, and they are characterized by the norm-attainment at a limit point of the Shilov boundary of $A$. In addition, a sufficient condition for the convex diametral local diameter two property of $A(K,X)$ is also provided. As a result, the similar results also hold for an infinite dimensional uniform algebra.
\end{abstract}

\maketitle

\section{Introduction}

 In this article, we study the diameter two properties of a uniform algebra and the vector-valued function space $A(K, (X, \tau))$ for a complex Banach space $X$ with a locally convex Hausdorff topology $\tau$ compatible with a dual pair by using the Urysohn-type lemma. The diameter two properties have gained a lot of attention recently and been studied in various Banach spaces. Exploring these properties involves weakly open subsets and slices of the unit ball $B_X$ of a Banach space $X$. Let $S_X$ be the unit sphere of a Banach space $X$ and $X^*$ be the dual space of $X$. For $x^* \in S_{X^*}$ and $\epsilon > 0$, a slice $S(x^*, \epsilon)$ of the unit ball $B_X$ is defined by $S(x^*, \epsilon) = \{x \in B_X : Re\, x^* x  > 1 - \epsilon\}$. Here we recall the definitions of the ``classical" diameter two properties introduced in \cite{ALN}.  
 \begin{Definition}(``Classical" D2Ps)
 	\begin{enumerate}[\rm(i)]
 		\item A Banach space $X$ has the strong diameter two property (SD2P) when every convex combination of slices has diameter two.
 		\item A Banach space $X$ has the diameter two property (D2P) when every nonempty relatively weakly open subset of the unit ball $B_X$ has diameter two.
 		\item A Banach space $X$ has the local diameter two property (LD2P) when every slice of the unit ball $B_X$ has diameter two.
 	\end{enumerate}
 \end{Definition}
 
 Since every weakly open subset of the unit ball contains a convex combination of slices \cite[Lemma II.1]{GGMS}, the SD2P implies the D2P. The D2P implies the LD2P because every slice is a relatively weakly open subset of the unit ball. It is well-known that these properties are weaker than the Daugavet property \cite[Theorem 4.4]{ALN} and that these properties are the other extremes of the Radon-Nikod\'ym property which states, geometrically, that there exists a nonempty slice of the unit ball with arbitrarily small diameter. The classical D2Ps are examined in various spaces. It is well-known that $c_0$ and $\ell_{\infty}$ has the SD2P. The infinite dimensional uniform algebra over a compact Hausdorff space also has the SD2P \cite{ALN, NW}, in fact, even stronger property called the symmetric strong diameter two property (SSD2P) \cite{ANP}. Other known examples with the D2P are interpolation spaces $L_1 + L_{\infty}$ and $L_1 \cap L_\infty$ equipped with certain norms \cite{AK}. These spaces do not satistfy the Daugavet property but the D2P. For a Banach space $X$, the space $C(K, (X, \tau))$ of $X$-valued continuous function over a compact Hausdorff space $K$, where $\tau$ a locally convex Hausdorff topology compatible to a dual pair, is also known to have the D2P \cite{BL}.
 
 Not only the ``classical" D2Ps mentioned above, but we also explore the ``diametral" version of the D2Ps. Studying these properties requires us to have a firm grasp on the $\Delta$-points and the Daugavet points. Let $\Delta_{\epsilon}(x) = \{y \in B_X : \|x - y\| \geq 2 -\epsilon\}$. Then a point $x \in S_X$ is said to be a $\Delta$-point if $x \in \overline{\text{conv}}\Delta_{\epsilon}(x)$ for every $\epsilon > 0$. We denote the set of all $\Delta$-points by $\Delta_X$. Similarly, a point $x \in S_X$ is said to be a Daugavet point if $B_X$ = $\overline{\text{conv}}\Delta_{\epsilon}(x)$ for every $\epsilon > 0$. Notice that every Daugavet point is a $\Delta$-point. We recall the definitions of the diametral diameter two properties introduced in \cite{AHLP, BLZ}.  

 \begin{Definition}(``Diametral" D2Ps)
 	\begin{enumerate}[\rm(i)]
 		\item A Banach space $X$ has the convex diametral local diameter two property (convex-DLD2P) if $B_X = \overline{\text{conv}}\Delta_X$, that is, the unit ball is the closed convex hull of the set of all $\Delta$-points.
 		\item A Banach space $X$ has the diametral local diameter two property (DLD2P) if for every slice $S$ of $B_X$, every $x \in S_X \cap S$, and every $\epsilon > 0$ there exists $y \in S$ such that $\|x - y\| \geq 2 - \epsilon$.  
 		\item A Banach space $X$ has the diametral diameter two property (DD2P) if for every relative weakly open subset $U$ of $B_X$, every $x \in S_X \cap U$, and every $\epsilon > 0$, there exists $y \in U$ such that $\|x-y\| \geq 2 - \epsilon$.
 	\end{enumerate}
 \end{Definition}
\noindent There is also a stronger version of the DD2P called the diametral strong diameter two property (DSD2P), which uses a convex combination of slices, but this property is now known to be equivalent to the Daugavet property \cite{K}. It was shown in \cite{BLZ} that the Daugavet property implies the DD2P and that the DD2P implies the DLD2P. The convex-DLD2P is stronger than the LD2P and weaker than DLD2P \cite{AHLP}. Moreover, we can see that the D2P is weaker than the DD2P from their definitions. If one wishes to see a clearer picture, we refer to a nice visualization of the relationship between various diameter two properties in the dissertation of Pirk (see \cite[pg 82-85]{P}). These properties have been also examined on several Banach spaces. For example, the spaces $c$ of convergent sequences and $l_\infty$ do not have the DLD2P but the convex-DLD2P \cite[Corollary 5.4, Remark 5.5]{AHLP}. Moreover, $c_0$ even fails to have the convex-DLD2P \cite[pg 18]{AHLP}.    

The DLD2P and the Daugavet property have equivalent definitions in terms of $\Delta$-points and Daugavet points. A Banach space $X$ having the DLD2P is equivalent to say that every point in $S_X$ is a $\Delta$-point (see \cite[Theorem 1.4]{IK} and \cite[Open problem (7)]{W}). Similarly, a Banach space $X$ has the Daugavet property if and only if every point in $S_X$ is a Daugavet point \cite[Corollary 2.3]{W}.

The organization of this article consists of four parts. In section 2, we provide necessary information about the uniform algebra $A$, the space $A(K, (X, \tau))$, and the Urysohn-type lemma, which is the key ingredient to prove our results throughout this article. In section 3, we make a few remarks on the symmetric strong diameter two property of a scalar-valued function algebra. In section 4, we show that the vector-valued function space $A(K, (X, \tau))$ over a uniform algebra also satisfies the D2P. By the Gelfand transform we show that a function space $A(\Omega, X)$ on a Hausdorff space $\Omega$ is isometrically isomorphic to a function space $A(M_A, (X^{**}, w^*))$ over a uniform algebra, so it is shown that $A(\Omega, X)$ also has the D2P if the base (function) algebra is infinite dimensional. This result shows D2P of some function spaces like $A_b(B_X, Y)$ and (resp. $A_u(B_X, Y)$) consisting of $Y$-valued functions which are bounded (resp. uniformly continuous) on $B_X$ and are holomorphic on the interior of $B_X$ when  $X$ and $Y$ are Banach spaces. In section 5, we present some results on the Daugavet and $\Delta$- points in a vector-valued function space $A(K, X)$ over a uniform algebra. Under the condition of $X$ being uniformly convex and the additional assumption that $A\otimes X\subset A(K, X)$, we show that these points are the same and characterized by norm-attainment at a limit point of the  Shilov boundary. Also under the same assumptions, a sufficient condition for the convex-DLD2P of $A(K, X)$ is provided.
 
 \section{Preliminaries}
 In this article, we assume that $K$ is a compact Hausdorff space and $X$, $Y$ are nontrivial complex Banach spaces unless specified. Let $X^*$ be the dual space of $X$. 
 We work with a Banach space $X$ equipped with a locally convex Hausdorff (LCH) topology $\tau$ compatible to the dual pair $(X, X^*)$  or $(X^*, X)$. Recall that $(X, \tau)$ is said to be compatible to a dual pair $(X, X^*)$ if $(X, \tau)^* = X^*$ and a LCH topology $\tau$ on $X^*$ is said to be compatible to a dual pair $(X^*, X)$ if $(X^*, \tau)^*=X$. In this paper, we say that a LCH topology $\tau$ on a Banach space $X$ is compatible to a dual pair if $(X, \tau)^* =X^*$ or $(X, \tau)^* = X_*$, where $X_*$ is the predual of $X$.
  
The weak topology on $X$ is denoted by $(X, w)$, the weak$^*$ topology on $X^*$ by $(X^*, w^*)$, and the norm topology on $X$ by $(X, \|\cdot\|)$. The Mackey topology for $X$ (resp. $X^*$) is the finest topology where every linear functional $x^* \in X^*$ (resp. linear functional defined by $x \in X$) is continuous. The Mackey-Arens theorem \cite[Theorem 8.7.4]{NB} says that a locally convex topology $\tau$ for $X$ (resp. $X^*$) compatible to a dual pair includes the weak (resp. weak$^*$-) topology and is included in the Mackey topology on $X$ (resp. $X^*$) and that is known to be coarser than the norm topology on $X$ \cite[IV.3.4]{SW} (resp. the norm topology on $X^*$ \cite[Example 8.5.5 and Example 8.8.9]{NB}).  Hence for any LCH topology $\tau$ compatible to a dual pair $(X, X^*)$ (resp. $(X^*, X)$), we have $(X, w) \subset \tau \subset (X, \|\cdot\|)$ (resp. $(X^*, w^*) \subset \tau \subset (X^*, \|\cdot\|)$). Also for such topology $\tau$ on $X$ compatible to $(X, X^*)$, $\tau$-bounded sets are weakly bounded and so norm-bounded. Similarly, for a LCH topology $\tau$ on $X^*$ compatible to $(X^*, X)$, $\tau$-bounded sets are weakly$^*$-bounded and so norm-bounded in $X^*$. For more details on the theory of topological vector spaces, we refer to \cite{NB, SW}. 

Now we show that the space $C(K, (X, \tau))$, which consists of all continuous functions from a compact Hausdorff space $K$ to $(X, \tau)$, equipped with the supremum norm is a Banach space. Even though this fact is already known in \cite{BL}, we include the proof for completeness.

\begin{proposition}\label{prop:ckxbanach}
	Let $K$ be a compact Hausdorff space, let $X$ be either a Banach space and let $\tau$ be a locally convex Hausdorff topology compatible to a dual pair. Then the space $C(K, (X, \tau))$ equipped with the supremum norm is a Banach space.
\end{proposition}

\begin{proof}
	
		Let us consider when $\tau$ is a LCH topology compatible to either a dual pair $(X, X^*)$ or $(X, X_*)$. Since $f(K)$ is $\tau$-compact for every $f \in C(K, (X, \tau))$, the range $f(K)$ is bounded in the norm topology of $X$. This makes the norm $\|f\| = \sup\{\|f(t)\|_X: t \in K\}$ to be well-defined.

		To show the completeness of $C(K, (X, \tau))$, let $(f_n)_{n = 1}^{\infty} \subset C(K, (X,\tau))$ be a Cauchy sequence. Then for any $\epsilon > 0$, there exists $N$ such that for every $n,m \geq N$, we have $\|f_n - f_m\| < \epsilon$. Notice that for every $t \in K$,
		\begin{eqnarray}\label{eq:cauchy}
			 \|f_n(t) - f_m(t)\|_X \leq \|f_n - f_m\| < \epsilon,
		\end{eqnarray}
	 and so $(f_n(t))_{n=1}^{\infty}$ is a Cauchy sequence in $X$ with respect to the norm topology. Since $X$ is complete with the norm topology, $(f_n(t))_{n=1}^{\infty}$ converges in norm. So for each $t \in K$, let $f(t) = \lim_{n\rightarrow \infty} f_n(t)$. Then by (\ref{eq:cauchy}), for any $\epsilon > 0$ there exists $N$ such that for all $n \geq N$, $\sup_{t \in K}\|f_n(t) - f(t)\|_X < \frac{\epsilon}{2}$.
	
	Now we have to show that $f \in C(K, (X, \tau))$. First, let $V$ be a $\tau$-neighborhood of $0$. Choose a balanced, convex $\tau$-neighborhood $W$ of $0$ such that $W + W + W \subset V$. Since $W$ is also open in the norm topology, there exists $\epsilon B_X \subset W$ and $N \in \mathbb{N}$ such that for all $n \geq N$ and $t \in K$, $f(t) - f_n(t) \in \epsilon B_X$. Hence
		\begin{align*}
		f(t) - f(t_0) &= (f(t) - f_N(t)) + (f_N(t) - f_N(t_0)) + (f_N(t_0) - f(t_0))	\\
		&\in \epsilon B_X + (f_N(t) - f_N(t_0)) + \epsilon B_X\\
		&\subset W + (f_N(t) - f_N(t_0)) + W. 
		\end{align*}
	Moreover, $f_N$ is continuous with respect to $\tau$. So there exists a neighborhood $U$ of $t_0$ such that for every $t \in U$, $f_N(t) - f_N(t_0) \in W$. Therefore, $f(t) - f(t_0) \in W + W + W \subset V$ for all $t \in U$. This shows that $f \in C(K, (X, \tau))$. Therefore, $C(K, (X, \tau))$ is a Banach space.     
\end{proof}

Let $\Omega$ be a Hausdorff space and $C_b(\Omega)$ be the Banach algebra of all bounded continuous functions over $\Omega$ equipped it the supremum norm. A function algebra $A(\Omega)$ on $\Omega$ is a closed subalgebra of $C_b(\Omega)$ that separates the points of $\Omega$ and contains constant functions. Here separating points means for each pair $(s,t) \in \Omega\times \Omega$ with $s\neq t$, there exists $f \in A(\Omega)$ such that $f(s) \neq f(t)$. 
If $\Omega$ is a compact Hausdorff space, a function algebra is called a uniform algebra.

We introduce $(X, \tau)$-valued function space over a uniform algebra for a Banach space $X$ as follows. 
\begin{definition}\label{def:funtionspace} Let $X$ be a Banach space and $\tau$ be a LCH topology on $X$ compatible to a dual pair. 
A closed subspace  $A(K, (X, \tau))$ of $C(K, (X, \tau))$ is said to be the  $(X, \tau)$-valued function space over a uniform algebra $A$ if the following two conditions are satisfied:
\begin{enumerate}[\rm(i)]
	\item The base algebra, defined by $A:= \{y^* \circ f: f \in A(K, (X, \tau)), y^* \in (X, \tau)^*\}$,  is a uniform algebra over $K$.
	\item $\phi f \in A(K, (X, \tau))$ for every $\phi\in A$ and  $f \in A(K, (X, \tau))$ where $(\phi f)(t) = \phi(t)f(t)$ for $t \in K$.   
\end{enumerate}
\end{definition}
\noindent Notice that if $X = \mathbb{C}$, we have the usual uniform algebra $A$ on $K$. When $\tau$ is the norm topology, we call the space $A(K, (X, \tau))$ an $X$-valued function space over a uniform algebra and denote the space by $A(K, X)$. Similar vector-valued version of a uniform algebra has been studied in \cite{CGKM}.

A point $t_0 \in K$ is said to be a strong boundary point for a uniform algebra $A$ if for every neighborhood $U$ containing $t_0$, there exists $f \in A$ such that $f(t_0) = \|f\|=1$ and $\sup_{t \in K \setminus U} |f(t)| < 1$. A subset $S \subset X$ is said to be a boundary if for each $f \in A$, there exists an element $t \in S$ such that $|f(t)| = \|f\|_\infty$. The Shilov boundary $\Gamma(A)$ of $A$ is the smallest closed boundary of $A$ and it is known that it is the intersection of all closed boundaries of $A$.  For a uniform algebra, let $K_A= \{\lambda \in A^*: \|\lambda\| = \lambda(1_A) = 1\}$ and denote the set of its extreme points by $\text{ex}\,K_A$. The set $\text{ex}\,K_A$ is called the Choquet boundary of $A$. It is well-known that each $\lambda \in \text{ex}\,K_A$ are associated with some elements $x \in K$ and  uniquely represented by the Dirac measure $\delta_x$ \cite[Lemma 4.3.2]{D}. For this reason, we use $\Gamma_0(A)$ to denote the set of such elements $x\in K$ corresponding to the elements in the Choquet boundary. Also, the set of strong boundary points and $\Gamma_0(A)$ coincide with each other when we consider a uniform algebra $A$ over a compact Hausdorff space $K$ \cite[Theorem 4.3.5]{D}. It is known that the Shilov boundary $\Gamma(A)$ of a uniform algebra $A$ is the closure of its Choquet boundary ${\Gamma_0(A)}$ \cite[Corollary 4.3.7.a]{D}.

Now we recall the Urysohn-type lemma that is extensively used throughout this article. 
 
 \begin{Lemma}\cite[Lemma 2.5]{CGK}\label{th:urysohn}
 Let $A \subset C(K)$ be a uniform algebra and $\Gamma_0$ be its Choquet boundary. Then, for every open set $U \subset K$ with $U \cap \Gamma_0 \neq \emptyset$ and $0< \epsilon < 1$, there exists $f \in A$ and $t_0 \in U \cap \Gamma_0$ such that $f(t_0) = \|f\|_{\infty} = 1$, $|f(t)| < \epsilon$ for every $t \in K \setminus U$ and
 \begin{equation}\label{eq:stolz}
 |f(t)| + (1 - \epsilon)|1 - f(t)| \leq 1\,\,\, \text{for all}\,\,\, t \in K.
 \end{equation}
\end{Lemma}

The following observation is useful for later. 

\begin{Lemma}\label{lem:AKXisom}
Let $X$ be a Banach space and let $\tau$ be a LCH topology compatible to a dual pair. Suppose that $A(K, (X, \tau))$ is a $(X, \tau)$-valued function space over the base algebra $A$ and $L$ is a closed boundary for $A$. The space of restrictions of elements of $A(K, (X, \tau))$ to $L$ is denoted by $A(L, (X, \tau))$ and the restrictions of elements of $A$ to $L$ is denoted by $A(L)$. Then $A(L, (X, \tau))$ is a $(X, \tau)$-valued function space over the base algebra $A(L)$ and  it is isometrically isomorphic to $A(K, (X, \tau))$.
\end{Lemma}
\begin{proof} Given $f\in A(K, (X, \tau))$, let $f_{|L}$ be the restriction of $f$ to $L$. Then we have  
\begin{align*}
\|f\| &= \sup \{ \|f(t)\| : t\in K \} = \sup \{ |x^*f(t)| : t\in K, x^*\in B_{(X, \tau)^*} \}\\
&=\sup\{ |x^* f(t)| : t\in L, x^*\in B_{(X, \tau)^*} \}=\sup\{ \|f(t)\| : t\in L \} =\|f_{|L}\|.
\end{align*} 
So $A(L, (X, \tau))$ is a surjective isometric image of $A(K, (X, \tau))$ by the map $f\mapsto f_{|L}$ and it is a closed subspace of $C(L, (X, \tau))$. It is easy to see that $A(L, (X, \tau))$ satisfies the condition (i) and (ii) in the Definition~\ref{def:funtionspace}  and its base algebra is $A(L)$. 
This completes the proof.
\end{proof}

Since the Shilov boundary $\Gamma$ of a uniform algebra is a closed boundary of $K$, by Lemma \ref{lem:AKXisom}  or  \cite[Theorem 4.1.6]{L}), a uniform algebra is isometric to a uniform algebra of its restrictions to $\Gamma$. We need another useful lemma on the isolated point of the Shilov boundary.

\begin{Lemma}\label{lem:auxiso}
	Let $A$ be a uniform algebra on a compact Hausdorff space $K$ and let $t_0$ be an isolated point of the Shilov boundary $\Gamma$ of $A$. Then there exists a function $\phi \in A$ such that $\phi(t_0) = \|\phi\| = 1$ and $\phi(t) = 0$ for $t \in \Gamma \setminus \{t_0\}$. 
\end{Lemma}

\begin{proof}
	Since $t_0$ is an isolated point of $\Gamma$, there is an open set such that $U\cap \Gamma=\{t_0\}$. The set of strong boundary points $\Gamma_0$ is dense in $\Gamma$. So the point $t_0 \in \Gamma$ is in fact a strong boundary point of $A$. By definition of a strong boundary point, there exists a function $\psi \in A$ such that $\psi(t_0) = \|\psi\| = 1$ and $\sup_{t \in K \setminus U }|\psi(t)| < 1$. Note that $\{(\psi_{|\Gamma})^n\}_{n = 1}^{\infty}$ converges uniformly in $C(L)$. By the isometry of the map from $f\in A$ to its restriction $f_{|L}$ to $L$, $\{\psi^n\}_{n=1}^\infty$ is Cauchy in $A$ and converges uniformly to $\phi$ in $A$.  Now it is clear that $\phi(t) = 0$ for $t \in \Gamma \setminus \{t_0\}$ and $\phi(t_0) = \|\phi\| = 1$. Thus, we obtain the desired function.   
\end{proof}

\section{A short remark on the SSD2P of function algebras}

First we start with what is known as the symmetric strong diameter two property (SSD2P) that is stronger than the SD2P.
\begin{definition}\cite[Definition 1.3]{ANP}
	A Banach space $X$ has the symmetric strong diameter two property (SSD2P) if whenever $n \in \mathbb{N}$, $S_1, \dots, S_n$ are slices of $B_X$, and $\epsilon > 0$, there exists $x_i \in S_i$, $i = 1,2,3,\dots, n$, and $y \in B_X$ such that $x_i \pm y \in S_i$ for every $i \in \{1, 2, 3, \dots, n\}$ and $\|y\| > 1 - \epsilon$.
\end{definition}
It has been already known from \cite[Theorem 4.2]{ALN} that the infinite dimensional uniform algebra over a compact Hausdorff space has the SSD2P and their proof method is extended to show the SSD2P of somewhat regular subspaces of the space $C_0(L)$ of continuous functions over a locally compact Hausdorff space $L$ vanishing at infinity \cite[Theorem 2.2]{ANP}.
\begin{definition}\cite[Definition 2.1]{ANP}
	Let $L$ be a locally compact Hausdorff space and $C_0(L)$ be the space of continuous functions over $L$ vanishing at infinity. A linear subspace $Y$ of $C_0(L)$ is somewhat regular, if whenever $V$ is a non-empty open subset of $L$ and $0 < \epsilon < 1$, there is an $f \in Y$ such that 
	\[
	\|f\| = 1 \,\,\,\text{and} \,\,\, |f(x)| \leq \epsilon \,\,\, \text{for every} \,\,\, x \in L\setminus V. 
	\]
\end{definition}
\noindent We make a few remarks here. In view of Lemma \ref{th:urysohn} we can easily verify that the infinite dimensional uniform algebra $A$ over a compact Hausdorff space $K$ is an example of somewhat regular subspace of $C(\Gamma)$ due to the fact that $A$ is isometric to $A(\Gamma)$ and that $\Gamma_0 = \Gamma_0(A)$ is dense in $\Gamma = \Gamma(A)$ because $K$ is compact. In addition, Lemma \ref{th:urysohn} is a special case of the Urysohn-type lemma for somewhat regular subspaces (see \cite[Lemma 2.4]{ANP}) that targets on uniform algebras. 
\begin{Theorem}\cite[Theorem 2.2]{ANP}\label{th:mainA}
The infinite dimensional uniform algebra $A$ over a compact Hausdorff space $K$ has the SSD2P.
\end{Theorem}

Now let $C_b(B_X)$ be the space of bounded, complex-valued, continuous functions on the unit ball $B_X$ of a Banach space $X$. The space $A_b(B_X)$ is a closed subalgebra of $C_b(B_X)$ that consists of holomorphic functions on the interior of $B_X$. The space $A_u(B_X)$ is a closed subalgebra of $C_b(B_X)$ that consists of functions in $A_b(B_X)$ that are uniformly continuous on $B_X$. From the fact that function algebras $A_b(B_X)$ and $A_u(B_X)$ are isometric to a uniform algebra on a compact Hausdorff space via the Gelfand transformation (see \cite[Proposition 2]{Lee} or Theorem~\ref{thm:Gelfand}), we can deduce the following fact on the SSD2Ps of $C_b(B_X)$, $A_b(B_X)$ and $A_u(B_X)$.

\begin{Corollary}
	Let $X$ be a Banach space. The function algebras $C_b(B_X)$,  $A_b(B_X)$ and $A_u(B_X)$ have the SSD2P. 
\end{Corollary}

\section{The D2P of vector-valued function spaces over a uniform algebra}

Now we consider the space $A(K,(X, \tau))$ where $X$ is a Banach space with a locally convex Hausdorff topology $\tau$ compatible with a dual pair. Here the norm of the space $A = A(K)$ is denoted by $\|\cdot\|_{\infty}$. First we recall a useful fact about the D2P of a direct sum of two Banach spaces.

\begin{Lemma}\cite[Lemma 2.2]{BL}\label{lem:inftysum}
	Let $X$ be a Banach space satisfying the diameter two property. Then for any arbitrary Banach space $Y$, $X \oplus_{\infty} Y$ has the diameter two property.
\end{Lemma}	

Here we present the main result of this section. 
\begin{Theorem}\label{th:d2pakx}
	Let $K$ be a compact Hausdorff space and $X$ be a Banach space endowed with a locally convex Hausdorff topology $\tau$ compatible to a dual pair. If the base algebra $A$ is infinite dimensioanl, then the space $A(K, (X, \tau))$ has the diameter two property, i.e. every nonempty weakly open subset of the unit ball has the diameter two.
\end{Theorem}

\begin{proof}
	In view of Lemma \ref{lem:AKXisom}, we only have to show the D2P for $A(\Gamma, (X, \tau))$ where $\Gamma = \Gamma(A)$ is a Shilov boundary of a uniform algebra. Since $A$ is infinite dimensional, $\Gamma$ is infinite. We write $A(\Gamma, X)$ instead of $A(\Gamma, (X, \tau))$ for convenience. 
	
	Now let $W$ be a nonempty weakly open subset of $A(\Gamma,X)$. Three cases will be considered.
	\begin{enumerate}[\rm(i)]
		\item Assume that $\Gamma$ is perfect. Let $a \in W \cap S_{A(\Gamma,X)}$.
		Then for every $\delta >0$, there exists $t_0 \in \Gamma$ such that $\|a(t_0)\|_X > 1 - \delta$. Also we can find $f \in S_{(X, \tau)^*}$ such that $\text{Re}\, f(a(t_0)) > 1 - \delta$ because $(X, \tau)^*= X^*$ or $(X, \tau)^* = X_*$, where $X_*$ is a predual of $X$.
		
		Define $U = \{t \in \Gamma : \text{Re}\, f(a(t)) > 1- \delta\}$. Notice that $U$ contains $t_0$. 
		Since $t_0$ is a limit point of $\Gamma$, we can construct pairwise disjoint open subsets $U_n \subset U$ for $n \in \mathbb{N}$. From the fact that $\Gamma_0 = \Gamma_0(A)$ is dense in $\Gamma$, we know $U_n \cap \Gamma_0\neq \emptyset$ for all $n \in \mathbb{N}$. So by Lemma \ref{th:urysohn}, for each $n \in \mathbb{N}$ there exist $\phi_n \in A$ and $t_n \in U_n \cap \Gamma_0$ such that $\phi_n(t_n) = \|\phi_n\|_{\infty} = 1$, $|\phi_n(t)| <  \frac{\delta}{2^{n+2}}$ for $t \in \Gamma \setminus U_n$, and
		\[
		|\phi_n(t)| + \left(1 - \frac{\delta}{2^{n+2}}\right)|1 - \phi_n(t)| \leq 1 \,\,\, \text{for} \,\,\, t \in \Gamma.
		\]
		Define $g_n = 1 - 2 \phi_n$. First we want to show that $\|g_n\| \leq 1 + \frac{\delta}{2^n}$. For every $t \in \Gamma$,
		\begin{eqnarray*}
			|1 - 2\phi_n(t)| &=& \left|1 - 2\phi_n(t)  + \frac{\delta}{2^{n+2}} - \frac{\delta}{2^{n+2}} + \frac{\delta}{2^{n+1}}\phi_n(t) - \frac{\delta}{2^{n+1}}\phi_n(t)\right|\\
			&\leq & \left|\left(1 - \frac{\delta}{2^{n+2}}\right)(1 - 2\phi_n(t))\right| +\frac{\delta}{2^{n+2}} + \left|\frac{\delta}{2^{n+1}}\phi_n(t)\right|\\
			&\leq& \left|\left(1 - \frac{\delta}{2^n}\right)(1 - \phi_n(t)) - \left(1 - \frac{\delta}{2^n}\right)\phi_n(t)\right| + \frac{\delta}{2^{n}}\\
			&\leq& 1 - |\phi_n(t)| + |\phi_n(t)| + \frac{\delta}{2^{n}}\\
			&\leq & 1 + \frac{\delta}{2^n}. 
		\end{eqnarray*}
		Hence $\|g_n\|_\infty \leq 1 + \frac{\delta}{2^n}$. Define $h_n = \frac{g_n}{1 +\delta/2^n}$. Notice that $h_n \in B_{A}$ for all $n \in \mathbb{N}$ and that $h_n(t)$ converges to $1$ pointwise. 
		
		Now, let $F \in A(\Gamma,X)^*$ be a bounded linear functional on $A(\Gamma,X)$. If we define $\varphi(h) = F(h a)$ for $h \in A$, $\varphi$ is a bounded linear functional on $A$ because $|\varphi(h)| \leq \|F\|_{A(\Gamma,X)^*}\|h\|_{\infty}\|a\| < \|h\|_{\infty}\|F\|_{A(\Gamma,X)^*} < \infty$ for all $h \in A$. By the Hahn-Banach extension theorem and the Riesz representation theorem, $\varphi$ is represented by a regular complex Borel measure $\mu$ and $\varphi(h_n) = \int_\Gamma h_n d\mu$. Moreover, $\varphi(h_n) =  \int_\Gamma h_n d\mu \rightarrow \int_{\Gamma} 1 d\mu = \varphi(1)$ by the bounded convergence theorem. Hence $F(h_n a) = \varphi(h_n) \rightarrow \varphi(1) = F(a)$ and so the functions $h_n a$ converges to $a$ weakly in $A(\Gamma,X)$. Furthermore, we have
		
		\begin{eqnarray*}
			\left\|a - h_na\right\| \geq \left\|a(t_n) - h_n(t_n)a(t_n)\right\|_X &=& \left|1 - h_n(t_n)\right|\|a(t_n)\|_X\\
			&\geq& \frac{2 + \delta/2^n}{1 + \delta/2^n} \cdot \text{Re}\, f(a(t_n))\\
			&\geq& \left(2-\frac{\delta/2^n}{1 + \delta/2^n}\right)(1 - \delta)\\
			&\geq& (2 - \delta)(1 - \delta)\\
			&\geq& 2 - 3\delta.
		\end{eqnarray*}
		Since $\delta>0$ is arbitrary, the diameter of $W\cap B_{A(\Gamma, X)}$ is $2$. The proof is finished.
		
		\item Now we consider when  $\Gamma$ has infinitely many isolated points.  Let $(t_n)_{n=1}^{\infty}$ be a sequence of isolated points and $t_0$ be the accumulation point of the sequence. Let $a \in W \cap S_{A(\Gamma,X)}$. Let $U$ be a $\tau$-open neighborhood around $0$ and let $V$ be an $\tau$-open neighborhood around $a(t_0)\neq 0$ such that $U \cap V = \emptyset$.  Notice that every $\tau$-open neighborhood is also an open neighborhood in the norm topology on $X$. So by the continuity of $a$, $a^{-1}(V)$ contains infinitely many $t_n$'s. Since $U$ is norm-open, for each $n \in \mathbb{N}$ there is $\delta>0$ such that $\delta B_X \subset U$. Hence $a(t_n)  \in V \subset X\setminus\delta B_X$. That is, $\|a(t_n)\|_X > \delta$ infinitely many $n \in \mathbb{N}$. So we may assume that $\|a(t_n)\|_X > \delta$ for all $n \in \mathbb{N}$. Since $\Gamma_0$ is dense in $\Gamma$, every isolated point $t_n \in \Gamma$ is a strong boundary point. Hence by Lemma \ref{lem:auxiso}, there exists $\phi_n \in A$ such that $\phi_n(t_n) = \|\phi_n\| = 1$ and $\phi_n(t)=0$ for all $t \in \Gamma\setminus\{t_n\}$.
		
		Define $g_n = 1 + \left(\frac{1}{\|a(t_n)\|_X}-1\right)\phi_n + \left(\frac{-1}{\|a(t_{n+1})\|_X} - 1\right)\phi_{n+1}$. We see that $g_n(t)$ converges to $1$ pointwise, $\sup_{n \in \mathbb{N}}\|g_n\|_{\infty} < 1 + \frac{1}{\delta}$ and $\|g_n a\| = 1$ for every $n \in \mathbb{N}$. Hence $g_n \rightarrow 1$ weakly in $A$. Moreover, $g_n a \rightarrow a$ weakly in $A(\Gamma,X)$ by the same argument in (i). When we compute the diameter of the relatively weakly open set of $B_{A(\Gamma, X)}$, we have 
		\begin{eqnarray*}
			\text{diam} \, (W \cap B_{A(\Gamma,X)}) \geq \|g_{n+1}a - g_n a\| &\geq& \|g_{n+1}(t_{n+1})a(t_{n+1}) - g_n(t_{n+1})a(t_{n+1})\|_X\\
			&=& \frac{2}{\|a(t_{n+1})\|_X}\|a(t_{n+1})\|_X = 2.
		\end{eqnarray*}
		
		\item Finally, let us consider when $\Gamma$ has only finitely many isolated points $t_1, t_2, \dots, t_n$. Denote $\Gamma_1= \{t_1, t_2, \dots, t_n\}$ and the perfect subset of $\Gamma$ by $\Gamma_2= \Gamma \setminus \Gamma_1$. Since the set of strong boundary points are dense in $\Gamma$, we see that $t_1, t_2, \dots, t_n$ are also strong boundary points and  for each $i = 1, 2, \dots, n$, once again by Lemma \ref{lem:auxiso}, there exist $\phi_i \in A(\Gamma)$ such that $\phi_i(t_i) = 1$ and $\phi_i(t) = 0$ for all $t \neq t_i$. Define a bounded linear operator $P: A(\Gamma, X) \rightarrow A(\Gamma, X)$ by $Pf = \sum_{i = 1}^n \phi_i f$. Then it is easy to check that $P$ is a norm-one projection.
Let $B = P(A(\Gamma, X))$ and $C$ be the restrictions of $f\in A(\Gamma, X)$ to $\Gamma_2$. Then $C$ is an $(X, \tau)$-valued function space over an infinite dimensional  uniform algebra on $\Gamma_2$. We will show that $A(\Gamma, X)$ is isometrically isomorphic to $B \oplus_{\infty} C$. Then by (i), $C$ has diameter two property and Lemma~\ref{lem:inftysum} completes the proof.

For $f\in A(\Gamma, X)$, define $\Phi(f) = (P(f), f|_{\Gamma_2})$. Then $\Phi:A(\Gamma, X)\to B\oplus_{\infty} C$ is well-defined and an isometry as follows
\[
\|\Phi(f)\| = \max\{\|P(f)\|, \|f_{\Gamma_2}\|\} = \max\{\sup_{t \in \Gamma_1}\|f(t)\|_X, \sup_{t\in {\Gamma_2}}\|f(t)\|_X\} = \sup_{t \in \Gamma} \|f(t)\|_X = \|f\|. 
\]
Given $(f, g)\in B\oplus_\infty C$, there exist $f_1, f_2 \in A(\Gamma, X)$ such that $f = P(f_1)$ and $g= f_2|_{\Gamma_2}$. Let $h = P(f_1)+ f_2-P(f_2)\in A(\Gamma, X)$ and it is clear that $\Phi(h) = (P(f_1), f_2|_{\Gamma_2}) = (f, g)$. So $\Phi$ is surjective. This shows that $A(\Gamma, X)$ is isometrically isomorphic to $B \oplus_{\infty} C$ and completes the proof.
	\end{enumerate}
\end{proof}

The norm, weak and weak-* toplogies are compatible to a dual pair and we get the following.

\begin{corollary}
	Let $X$ be a Banach space and let $K$ be an infinite compact Hausdorff space. Then the spaces $A(K, X)$, $A(K, (X,w))$, and $A(K, (X^*,w^*))$ have the D2P if their corresponding base algebras are infinite dimensional.
\end{corollary}

Now we introduce a Banach space $X$-valued function space (over a function space) on a Hausdorff space and we show that this pace is isometrically isomorphic to a $(X^{**}, w^*)$-valued function space over a uniform algebra.  Let $\Omega$ be a Hausdorff space and let $X$ be a Banach space equipped with the norm topology. The space $C_b(\Omega, X)$ is the Banach space of all bounded $X$-valued continuous functions over $\Omega$ equipped it the supremum norm. 

\begin{definition}\label{def:functionspacevector}
Let $X$ be a Banach space. An $X$-valued function space over a function algebra on $\Omega$,  denoted by $A(\Omega, X)$, is a closed subspace of $C_b(\Omega, X)$ satisfying
\begin{enumerate}[\rm(i)]
	\item the base algebra $A := \{x^* \circ f: f \in A(\Omega, X)\}$ is a function algebra on $\Omega$.
	\item $fg \in A(\Omega,X)$ for every $f \in A$ and $g \in A(\Omega, X)$.
\end{enumerate}
\end{definition}
Given a function algebra $A$ on a Hausdorff space $\Omega$, let  $M_A$ be its maximal ideal space consisting of all nonzero algebraic homomorphisms from $A$ to $\mathbb{C}$. $M_A$ is a compact Hausdorff space with the Gelfand topology \cite[Theorem 11.9]{R}. The Gelfand transform $\hat{f}: M_A \rightarrow \mathbb{C}$ of $f \in A$ is defined by $\hat{f}(\phi) = \phi(f)$. Now for $g \in A(\Omega,X)$ and $x^* \in X^*$ let $\hat{g}(\phi)(x^*) = \phi(x^* \circ g)$. From the fact that
\[
|\hat{g}(\phi)(x^*)| = |\phi(x^*\circ g)| \leq \| x^* \circ g\|_{\infty} = \sup_{t \in \Omega}\{|x^*(g(t))|\} \leq \|g\| \|x^*\|_{X^*},
\] 
the mapping $\hat{g}(\phi)$ is a bounded linear functional on $X^*$. Moreover, $\hat{g} : M_A \rightarrow X^{**}$ is continuous on $M_A$ if we consider the weak$^*$-topology on $X^{**}$. Let $A(M_A, (X^{**}, w^*))$ be the set of such $\hat{g}$'s. 

\begin{Theorem}\label{thm:Gelfand}
Let $X$ be a Banach space and $A(\Omega, X)$ be an $X$-valued function space over a base algebra $A$ on a Hausdorff space $\Omega$. Then $A(\Omega, X)$ is isometrically isomorphic to an $(X^{**}, w^*)$-valued function space over a base algebra $\hat{A}$ on the compact Hausdorff space $M_A$. In fact, $M_A$ is the maximal ideal space of $A$ with the Gelfand topology. 
\end{Theorem}

We split the proof of Theorem \ref{thm:Gelfand} into several parts for readability. First we show that the mapping $g\mapsto \hat{g}$ is an isometry.
\begin{proposition}
Let $X$ be a Banach space, let $A(\Omega, X)$ be an $X$-valued function space over a base algebra $A$ on $\Omega$ and let $M_A$ be the maximal ideal space of $A$. Then for every $g \in A(\Omega, X)$ and $x^* \in X^*$, the mapping $g \mapsto \hat{g}$ where $\hat{g} \in A(M_A, (X^{**}, w^*))$ is defined by $\hat{g}(\phi)(x^*) = \phi(x^* \circ g)$ is an isometry. 
\end{proposition}
\begin{proof}
	By the fact that $\|\phi\|_{\infty} = 1$ we obtain
	\begin{align*}
		\|\hat{g}\| = \sup_{\phi \in M_A}\{\|\hat{g}(\phi)\|\} &= \sup_{\phi \in M_A}\sup_{ x^* \in B_{X^*}}\{|\hat{g}(\phi)(x^*)|\} = \sup_{\phi \in M_A}\sup_{ x^* \in B_{X^*}}\{|\phi(x^*\circ g)\}\\
		&\leq \sup_{ x^* \in B_{X^*}}\{\|x^*\circ g\|_{\infty}\}\\
		&= \sup_{x^* \in B_{X^*}} \sup_{t \in \Omega} \{|(x^* \circ g)(t)|\}\\
		&\leq \sup_{t \in \Omega} \{\|g(t)\|_X\} = \|g\|.
	\end{align*}
	To show the reverse inequality, let $\phi_t \in M_A$ be an evaluation functional at $t \in \Omega_0 = \{t \in \Omega : g(t) \neq 0 \}$. Then we have
	\begin{align*}
		\|g\| = \sup_{t \in \Omega_0} \{\|g(t)\|_X\}  &= \sup_{t \in \Omega_0} \sup_{x^* \in B_{X^*}}\{|(x^* \circ g(t)|\} = \sup_{t \in \Omega_0}\sup_{ x^* \in B_{X^*}}\{|\phi_t(x^*\circ g)\}\\
		&\leq \sup_{\phi \in M_A}\sup_{ x^* \in B_{X^*}}\{|\phi(x^*\circ g)\}\\
		&=\sup_{\phi \in M_A}\sup_{ x^* \in B_{X^*}}\{|\hat{g}(\phi)(x^*)|\}\\
		&\leq  \sup_{\phi \in M_A}\{\|\hat{g}(\phi)\|\} = \|\hat{g}\|,
	\end{align*}
	so the mapping $g \mapsto \hat{g}$ is an isometry.
\end{proof}

Now we show that the space $A(M_A, (X^{**}, w^*))$ satisfies the condition (i) and (ii) in the Definition~\ref{def:funtionspace} for a $(X, \tau)$-valued function space over a uniform algebra on a compact Hausdorff space. 

\begin{proposition}\label{prop:conduno}
Let $X$ be a Banach space, let $A(\Omega, X)$ be a $X$-valued function space over a base algebra $A$ on $\Omega$, and let $M_A$ be the maximal ideal space of $A$. Then the space \[A_1 = \{x^* \circ \hat{g}: \hat{g} \in A(M_A, (X^{**}, w^*))\,\,\, \text{and} \,\,\, x^* \in X^*\}\]  is a uniform algebra over $M_A$.
\end{proposition}

\begin{proof}
From the definition of the Gelfand transformation, notice that $x^*\circ \hat{g}(\phi) = \phi(x^* \circ g) = \widehat{x^* \circ g}(\phi)$. Hence we can rewrite $A_1 = \{\widehat{x^*\circ g}: g \in A(\Omega, X)\,\,\, \text{and} \,\,\, x^* \in X^*\}$. We see that the set $A_1$ is, in fact, the image of the Gelfand transformation on $A$. Thus $A$ is isometrically isomorphic to $A_1$ and so $A_1$ is a closed subalgebra of $C(M_A)$ that separates the points of $M_A$ and contains constant functions. This shows that $A_1$ is a uniform algebra over $M_A$.
\end{proof}

\begin{proposition}\label{prop:condtres}
Suppose that $X$ is a Banach space and  $A(\Omega, X)$ is an $X$-valued function space over a base algebra $A$ on $\Omega$. Let $M_A$ be the maximal ideal space of $A$ and let
$$A_1 = \{x^* \circ \hat{g}: \hat{g} \in A(M_A, (X^{**}, w^*))\,\,\, \text{and} \,\,\, x^* \in X^*\}.$$
Then $\phi \cdot \hat{g} \in A(M_A, (X^{**}, w^*))$ for every $\phi \in A_1$ and $ \hat{g} \in A(M_A, (X^{**}, w^*))$. 
\end{proposition}

\begin{proof}	
Let $\hat{f}, \hat{g} \in A(M_A, (X^{**}, w^*))$ and let $x^* \circ \hat{f} \in A_1$ where $x^* \in (X^{**}, w^*)^* = X^*$. We claim that $(x^*\circ \hat{f}) \cdot \hat{g} \in A(M_A, (X^{**}, w^*))$. Note first that $h= (x^*\circ f) \cdot g \in A(\Omega, X)$. We need to show that $\hat h = (x^*\circ \hat{f}) \cdot \hat{g}$.

Since $\widehat{x^* \circ f}(\phi) = x^* \circ \hat{f}(\phi) = \phi(x^* \circ f)$ for every nonzero algebra homomorphism $\phi \in M_A$, we have 
\[\hat h (\phi)(y^*) = (\widehat{x^* \circ f} \cdot \hat{g})(\phi)(y^*) = (x^* \circ \hat f)(\phi) \cdot \hat{g}(\phi)(y^*) =( (x^*\circ \hat f)\cdot \hat g) (\phi)(y^*)\]
for $y^* \in X^*$. Therefore, $\hat h = (x^*\circ \hat{f}) \cdot \hat{g}$ and this proves our claim.        
\end{proof}

Theorem~\ref{th:d2pakx} and Theorem~\ref{thm:Gelfand} prove the following.

\begin{corollary}
Suppose that $X$ is a Banach space and $A(\Omega, X)$ is an $X$-valued function space over a function algebra $A$ on a Hausdorff space $\Omega$. If the base algebra $A$ is infinite dimensional, then $A(\Omega, X)$ has the D2P.
\end{corollary}
For Banach spaces $X$ and $Y$, let $C_b(B_X, Y)$ be the bounded $Y$-valued continuous functions on the unit ball $B_X$. The space $A_b(B_X, Y)$ is defined to be a closed subspace of  $C_b(B_X, Y)$ that consists of holomorphic funcions on the interior of $B_X$ and  the space $A_u(B_X, Y)$ is a closed subspace of $A_b(B_X, Y)$ consisting of uniformly continuous functions on $B_X$. It is easy to check that $C_b(B_X, Y)$, $A_b(B_X, Y)$ and $A_u(B_X, Y)$ are $Y$-valued function spaces  in Definition~\ref{def:functionspacevector}. Hence we prove the following.

\begin{corollary}
Let $X$ and $Y$ be Banach spaces. Then $C_b(B_X, Y)$, $A_b(B_X, Y)$ and $A_u(B_X, Y)$ have the D2P.
\end{corollary}  

\section{Daugavet points and $\Delta$-points on vector-valued function spaces over a uniform algebra}

In this section, we present another application of Lemma \ref{th:urysohn} to the Daugavet points and $\Delta$-points of $A(K,X)$ with respect to its Shilov boundary of the base algebra. Recall that a Banach space $X$ has the Daugavet property if for every rank-one operator $T:X \rightarrow X$ satisfies $\|I + T\| = 1 + \|T\|$. Well-known examples are $L_1(\mu)$ and $L_{\infty}(\mu)$ where $\mu$ is a nonatomic measure \cite{FS, Lo} and $C(K)$ where $K$ does not have isolated points \cite{Dau}. An infinite dimensional uniform algebra over a compact Hausdorff space satisfies the Daugavet property if the Shilov boundary does not contain isolated points \cite{W2, Wo}. More generally, if a locally compact Hausdorff space $L$ does not have isolated points, somewhat regular subspaces of $C_0(L)$ also have the Daugavet property \cite{ANP}. Similar results are extended to vector-valued spaces. The space $C(K, X)$ has the Daugavet property if $K$ does not have isolated points \cite{W}. A subspace $A^X$ of $C(K, X)$, defined by $A^X = \{f \in C(K, X): x^* \circ f \in A\}$ where $A$ is a uniform algebra, is known to satisfy the Daugavet property from the fact that it satisfies the polynomial Daugavet property \cite{CGKM}. 

We mentioned earlier that we can investigate the Daugavet property and the DLD2P of Banach spaces through the Daugavet and $\Delta$-points. So it is clear that the nonexistence of isolated points in the Shilov boundary of an infinite dimensional uniform algebra implies that every point on the unit sphere $S_A$ is a Daugavet point. However, a characterization of Daugavet points and $\Delta$-points on an infinite dimensional uniform algebra has not been known. In \cite{AHLP}, the $\Delta$-points and the Daugavet points for $C(K)$ are determined by norm-attainment at a limit point of $K$. Motivated by this result, we provide a characterization for the Daugavet points and the $\Delta$-points on the vector-valued function space $A(K, X)$ over a uniform algebra when $X$ is uniformly convex. First we have the following observation on the $\Delta$-points.

\begin{Lemma}\cite[Lemma 2.1]{AHLP}\label{lem:deltapoint}
	Let $X$ be a Banach space. The following are equivalent:
	\begin{enumerate}[\rm(i)]
		\item $x \in S_X$ is a $\Delta$-point.
		\item for every $x^* \in X^*$ with $x^*x = 1$, the projection $P = x^* \otimes x$ satisfies $\|I - P\| \geq 2$.
	\end{enumerate}
\end{Lemma}

We show that if $f \in S_{A(K, (X, \tau))}$ attains its norm at a limit point on the Shilov boundary of its base algebra with the additional condition $A\otimes X\subset A(K, (X, \tau))$, it is a Daugavet point. For a function algebra $A$ on $\Omega$ and a Banach space $X$, $A\otimes X$ means the set of all functions $f\otimes x$ ($f\in A$, $x\in X$) defined by $(f\otimes x)(t) = f(t)x$ ($t\in \Omega$).

\begin{Theorem}\label{th:dpointakx}
Let $K$ be a compact Hausdorff space and let $X$ be a Banach space endowed with a locally convex Hausdorff topology $\tau$ compatible to a dual pair. Suppose that a function space $A(K, (X, \tau))$ over a base algebra $A$ satisfies the additional condition that $A\otimes X\subset A(K, (X, \tau))$ and let $\Gamma$ be the Shilov boundary of $A$. If  $f$ is a norm-one element of  $A(K, (X, \tau))$ and there is a limit point $t_0$ of $\Gamma$ such that $\|f\| = \|f(t_0)\|_X$, then $f$ is a Daugavet point.
\end{Theorem}

\begin{proof}
Suppose that $\|f\|=\|f(t_0)\|_X = 1$ for a limit point $t_0$ of $\Gamma$. Since $\Gamma$ is closed, $t_0$ is an element of $\Gamma$.
Let $A(\Gamma, (X, \tau))$ be set of all restrictions of $f\in A(K, (X, \tau))$ to the Shilov boundary $\Gamma$.
By Lemma~\ref{lem:AKXisom}, the mapping, defined by $g\in A(K, (X, \tau))\mapsto g|_\Gamma$, is an isometry. So it is enough to show that $f|_\Gamma$ is a Daugavet point of $A(\Gamma, (X, \tau))$ and we may assume that $K=\Gamma$. Denote the base algebra of $A(K, (X, \tau))$ as $A$.

Now fix $g \in B_{A(K, (X, \tau))}$ and  $\epsilon > 0$. Let $U$ be an open set containing $t_0$ such that $\|f(t) - f(t_0)\|_X < \epsilon$ for every $t \in U$. By the Hausdorff condition and the fact that $t_0$ is a limit point of $\Gamma$, we can find pairwise disjoint nonempty open subsets $U_1, U_2, \dots$ such that $U_n \subset U$ for each $n \in \mathbb{N}$. Since the set of strong boundary points is dense in $\Gamma$, each $U_n$ contains a strong boundary point. By Lemma \ref{th:urysohn}, for each $n\in \mathbb{N}$, there exist $\phi_n \in A$ and $t_n \in U_n \cap \Gamma_0$ such that $\phi_n(t_n) = \|\phi_n\|_{\infty} = 1$, $|\phi_n(t)| < \frac{\epsilon}{3\cdot2^n}$ for all $t \in K \setminus U_n$, and
\[
	|\phi_n(t)| + \left(1 - \frac{\epsilon}{3\cdot2^n}\right)|1 - \phi_n(t)| \leq 1\,\,\, \text{for all}\,\,\, t \in K.
\]
Now for each $n\in\mathbb{N}$, define a function $g_n(t) = (1-\frac{\epsilon}{3\cdot2^n})(1- \phi_n(t))g(t) - \phi_n(t)f(t_0)$. Then $g_n$ is an element of  $A(K, (X, \tau))$ and we show that $g_n \in B_{A(K, (X, \tau))}$ for all $n \in \mathbb{N}$. It is easy to see that for all $t \in K$,
	\[
	\|g_n(t)\|_X = \left\|(1-\frac{\epsilon}{3\cdot2^n})(1- \phi_n(t))g(t) - \phi_n(t)f(t_0)\right\|_X \leq \left(1 - \frac{\epsilon}{3\cdot2^n}\right)| 1-\phi_n(t)| + |\phi_n(t)| \leq 1.   
	\]
Hence $\|g_n\| \leq 1$. Notice that $g_n(t_n) = -f(t_0)$. This implies that 
	\[
	\|f - g_n\| \geq \|f(t_n) - g_n(t_n)\|_X = \|f(t_n) + f(t_0)\|_X > 2\|f(t_0)\|_X - \|f(t_n) - f(t_0)\|_X \geq  2 - \epsilon,  
	\]
and so $g_n \in \Delta_{\epsilon}(f)$. Furthermore, for every $t \in U_n$ we see that
	\begin{eqnarray*}
		\|g(t) - g_n(t)\|_X &=& \left\|g(t) - \left(1-\frac{\epsilon}{3\cdot2^n}\right)(1- \phi_n(t))g(t) + \phi_n(t)f(t_0)\right\|_X\\
		&=& \left\|\frac{\epsilon}{3\cdot2^n} g(t) + \left(1 -\frac{\epsilon}{3\cdot2^n}\right)\phi_n(t)g(t) + \phi_n(t)f(t_0)\right\|_X\\
		&\leq& \frac{\epsilon}{3\cdot2^n} + \left(1 - \frac{\epsilon}{3\cdot2^n}\right) + 1 = 2.
	\end{eqnarray*}
Since $|\phi_n(t)| < \frac{\epsilon}{3\cdot2^n}$ for $t \in K \setminus U_n$, we also have
\begin{align*}
	\|g(t) - g_n(t)\|_X &= \left\|\frac{\epsilon}{3\cdot2^n} g(t) + \left(1 -\frac{\epsilon}{3\cdot2^n}\right)\phi_n(t)g(t) + \phi_n(t)f(t_0)\right\|_X \\
	&\leq \frac{\epsilon}{3\cdot2^n} + \left(1 - \frac{\epsilon}{3\cdot2^n}\right)\frac{\epsilon}{3\cdot2^n} + \frac{\epsilon}{3\cdot2^n} < \frac{\epsilon}{2^n}.
\end{align*}
From the fact that each $U_n$'s are pairwise disjoint, we obtain, for each $t \in \cup_{n = 1}^m U_n$,
	\[
	\left\|g(t) - \frac{1}{m}\sum_{n = 1}^m g_n(t)\right\|_X \leq \frac{1}{m}\left(2 +\left(\frac{\epsilon}{2} + \frac{\epsilon}{4} + \dots + \frac{\epsilon}{2^m}\right)\right) = \frac{1}{m}\left(2 + \left(1 - \frac{1}{2^m}\right)\epsilon\right)   
	\]
and for each $t \in K \setminus \cup_{n = 1}^m U_n$, we have
	\[
	\left\|g(t) - \frac{1}{m}\sum_{n = 1}^m g_n(t)\right\|_X \leq   \frac{1}{m} \left(\frac{\epsilon}{2} + \frac{\epsilon}{4} + \dots + \frac{\epsilon}{2^m}\right)= \frac{1}{m}  \left(1 - \frac{1}{2^m}\right)\epsilon.
	\] 
Hence,  we get
	\[
	\left\|g - \frac{1}{m}\sum_{n = 1}^m g_n\right\|\leq \frac{1}{m}\left(2 + \left(1 - \frac{1}{2^m}\right)\epsilon\right). 
	\]
By taking the limit $m\to \infty$, we show that $g \in \overline{\text{conv}}\Delta_\epsilon(f)$. Hence $B_X = \overline{\text{conv}}\Delta_\epsilon(f)$,  in other words, $f$ is a Daugavet point.
\end{proof}

From this, we obtain a sufficient condition for the Daugavet property of $A(K, X)$ since the map $t\mapsto \|f(t)\|$ is continuous on a compact space and attains its maximum on $K$ if $\tau$ is norm topology and $f\in A(K, X)$. 

\begin{Corollary}\label{cor:daugavetakx}
Let $K$ be a compact Hausdorff space and let $X$ be a Banach space. Suppose that a function space $A(K, X)$ over a base algebra $A$ satisfies the additional condition that $A\otimes X\subset A(K, X)$. If the Shilov boundary $\Gamma$ of the base algebra $A$ does not have isolated points, then $A(K, X)$ has the Daugavet property.
\end{Corollary}

Furthermore, when $X$ is uniformly convex, Daugevet points and $\Delta$-points are the same in $A(K, X)$ and characterized by the norm-attainment at a limit point of $\Gamma$. Recall that the modulus of convexity $\delta_X(\epsilon)$ of a Banach space $X$ is defined by for $0<\epsilon<2$ 
\[ \delta_X(\epsilon) =\inf \left\{ 1- \left\|\frac{ x+y}2 \right\| : x,y\in B_X, \|x-y\|\ge \epsilon \right\}.\]
A Banach space $X$ is said to be uniformly convex if $\delta_X(\epsilon)>0$ for each $0<\epsilon<2$. 

\begin{Theorem}\label{th:deqakx}
	For a uniformly convex Banach space $X$ and a compact Hausdorff space $K$, let $\Gamma$ be the Shilov boundary of the base algebra $A$ of a function space $A(K, X)$ and let $f \in S_{A(K, X)}$.  Then (i)$\implies$ (ii) $\implies$ (iii) holds. 	
	         \begin{enumerate} [\rm(i)]
		\item $f$ is a Daugavet point.
		\item $f$ is a $\Delta$-point.
		\item there is a limit point $t_0$ of $\Gamma$ such that $\|f\| = \|f(t_0)\|_X$. 
	\end{enumerate}
	Moreover, if we assume the additional condition that $A\otimes X\subset A(K, X)$, then (i), (ii) and (iii) are equivalent.
\end{Theorem}

\begin{proof}
	In view of Lemma \ref{lem:AKXisom}, $A(K, X)$ is isometric to $A(\Gamma, X)$ and so it is enough to prove our claim with respect to $f_{|\Gamma} \in S_{A(\Gamma, X)}$. So we may assume that $K = \Gamma$. (i) $\implies$ (ii) is clear from their definitions. (iii) $\implies$ (i) is already shown in Theorem \ref{th:dpointakx} with the additional assumption that $A\otimes X\subset A(K, X)$. Hence we only need to show (ii) $\implies$ (iii). 
	
	
	Assume to the contrary that $f \in S_{A(K, X)}$ is a $\Delta$-point but $\|f\| \neq \|f(t)\|_X$ for every limit point of $K$. Let $F = \{ t \in K : \|f(t)\|_X = 1\}$. Then $F$ only contains isolated points of $K$ by the assumption and is nonempty. Since $K$ is compact, we also see that $|F| < \infty$. For each $t \in F$ we can always find $x_t^* \in S_{X^*}$ such that $x_t^*(f(t)) = 1$. 
	
	Since $X$ is uniformly convex, let  $\delta = \delta_X\left(\frac1{2|F|}\right)>0$, where $|F|$ is the number of elements of $F$. If $\text{Re}\, x_t^*(x) \geq 1 - \delta$ and $x\in B_X$, then
$\left \| \frac{x+f(t)}2 \right\| \geq \text{Re}\, x_t^*\left( \frac{x+f(t)}2  \right) \ge 1-\frac{1}2 \delta > 1-\delta_X\left(\frac 1{2|F|}\right)$. This means that $\|x-f(t)\|< \frac 1{2|F|}$ by the definition of $\delta_X$. Hence
	\begin{equation}\label{eq:uniconv}
		\text{if}\,\,\, \text{Re}\, x_t^*(x) \geq 1 - \delta\,\, \text{and} \, x\in B_X\,\,\text{then} \,\,\, \|x - f(t)\|_X \leq \frac{1}{2|F|} \,\,\, \text{for every} \,\,\,t \in F.
	\end{equation}
	
	Let $\epsilon = 1 - \max_{t \in K \setminus F}\|f(t)\|_X>0$. Define a bounded linear functional $\psi$ on $A(K, X)$ by 
	\[\psi(g) = \frac{1}{|F|}\sum_{t \in F} x_t^*(g(t)),\] where $g \in A(K, X)$. Notice that $\|\psi\| = \psi(f) = 1$ and this defines a bounded projection $P: A(K, X) \rightarrow A(K, X)$ by $P(g) = \psi(g)f$ for $g \in A(K, X)$. It is clear that $\|P\|=\|P(f)\| = \|f\| = 1$. We see that $P$ is in fact a norm-one projection on $A(K, X)$. Then in view of  Lemma~\ref{lem:deltapoint}, we should have $\|I - P\|\ge 2$. But we show that in fact $\|I-P\|<2$.
	
	For a given $g \in B_{A(K, X)}$, if $t \in K \setminus F$ we have
	\begin{equation}\label{eq: outF}
		\|g(t) - Pg(t)\|_X = \|g(t) - \psi(g)f(t)\|_X \leq 1 + \|f(t)\|_X \leq 2 - \epsilon.
	\end{equation}
	Now we divide $F$ into two parts. Let $t_0 \in F_1 = \{t \in F : |x_t^*(g(t))| \geq 1 - \delta\}$. Then there exists a scalar $\lambda_0 \in B_{\mathbb{C}}$ such that $|x_{t_0}^*(g(t_0))| = \lambda_0 x_{t_0}^*(g(t_0)) = x_{t_0}^*(\lambda_0 g(t_0)) \geq 1 - \delta$. Then by (\ref{eq:uniconv}), we have $\|\lambda g(t_0) - f(t)\|_X \leq \frac{1}{2|F|}$. Then
	\begin{eqnarray*}
		\|g(t_0) - Pg(t_0)\|_X &=& \left\|g(t_0) - \frac{1}{|F|}\sum_{t \in F} x_t^*(g(t)) f(t_0)\right\|_X \leq \left\|g(t_0) - \frac{1}{|F|}x_{t_0}^*(g(t_0)) f(t_0)\right\|_X + \frac{|F|-1}{|F|}\\
		&=& \left\|\lambda_0 g(t_0) - \frac{1}{|F|}x_{t_0}^*(\lambda_0 g(t_0)) f(t_0)\right\|_X + \frac{|F|-1}{|F|}.
	\end{eqnarray*}
	
From the fact that $x_{t_0}^*(\lambda_0 g(t_0)) \leq 1$, we have $1- \frac{1}{|F|}x_{t_0}^*(\lambda_0 g(t_0)) \geq 0$. Hence
	\begin{eqnarray*}
		\left\|\lambda_0 g(t_0) - \frac{1}{|F|}x_{t_0}^*(\lambda_0 g(t_0)) f(t_0)\right\|_X  &\leq& \|\lambda_0 g(t_0) - f(t_0)\|_X + \|f(t_0)\|_X\left(1- \frac{1}{|F|}x_{t_0}^*(\lambda_0 g(t_0))\right)\\
		&\leq& \frac{1}{2|F|} + 1 - \frac{1}{|F|}(1 - \delta),
	\end{eqnarray*}
	and so
	
	\begin{equation}\label{eq:Fpt1}
		\|g(t_0) - Pg(t_0)\|_X  \leq \frac{1}{2|F|} + 1 - \frac{1}{|F|}(1 - \delta) + \frac{|F|-1}{|F|} \leq 2 - \frac{1}{2|F|}.
	\end{equation}
	
	When $t_0 \in F_2 = F \setminus F_1$, we have $|x_{t_0}^*(g(t_0))| < 1 - \delta$. So we see that
	\begin{eqnarray*}
		\|g(t_0) - Pg(t_0)\|_X &\leq& \left\|\lambda_0 g(t_0) - \frac{1}{|F|}x_{t_0}^*(\lambda_0 g(t_0)) f(t_0)\right\|_X + \frac{|F|-1}{|F|}.\\
		&\leq& 1 + \frac{1}{|F|}(1 - \delta) + \frac{|F| - 1}{|F|} \leq 2 - \frac{\delta}{|F|}.
	\end{eqnarray*}
	Hence, with (\ref{eq: outF}) and (\ref{eq:Fpt1}) we obtain $\|g - Pg\| \leq \max\{2- \frac{1}{2|F|}, 2 - \frac{\delta}{|F|}, 2 - \epsilon \} < 2$. 
	Since $g \in B_{A(K, X)}$ is chosen arbitrarily,  $\|I-P\|<2$, and we get a contradiction. 
\end{proof}

Since $\mathbb{C}$ is uniformly convex, we have the following consequences for a scalar-valued uniform algebra. 
\begin{Corollary}\label{th:dpoint}
	Let $K$ be a compact Hausdorff space and let $\Gamma$ be the Shilov boundary of $A(K)$. Then for $f \in S_{A(K, X)}$, the following are equivalent:
		\begin{enumerate}[\rm(i)]
			\item $f$ is a Daugavet point.
			\item $f$ is a $\Delta$-point.
			\item  there is a limit point $t_0$ of $\Gamma$ such that $\|f\|_\infty = |f(t_0)|$.  
		\end{enumerate}			
\end{Corollary}

By the characterization given in Theorem \ref{th:deqakx}, we show when the Daugavet property, DD2P, and DLD2P of $A(K, X)$ are equivalent to each other.  
 
\begin{Theorem}
 	Let $X$ be a uniformly convex Banach space, let $K$ be a compact Hausdorff space, and let $\Gamma$ be the Shilov boundary of the base algebra of $A(K, X)$. Then (i) $\implies$ (ii) $\implies$ (iii) $\implies$ (iv) holds:
 	\begin{enumerate}[\rm(i)]
 		\item $A(K,X)$ has the Daugavet property. 
 		\item $A(K,X)$ has the DD2P.
 		\item $A(K,X)$ has the DLD2P.
 		\item The Shilov boundary $\Gamma$ does not have isolated points.
 	\end{enumerate}
	Moreover, if we assume the additional condition that $A\otimes X\subset A(K, X)$, then (i), (ii), (iii) and (iv) are equivalent.
 \end{Theorem}
 
 \begin{proof}
 	 Since $A(K, X)$ is isometric to $A(\Gamma, X)$ by Lemma \ref{lem:AKXisom}, showing the equivalence for $A(\Gamma, X)$ is enough. It is a well-known that (i) $\implies$ (ii) $\implies$ (iii) \cite{BLZ}. 
 	By Corollary \ref{cor:daugavetakx}, we have (iv) $\implies$ (i) with the additional assumption  that $A\otimes X\subset A(K, X)$.
 	
 	Finally we prove  (iii) $\implies$ (iv). Assume that $A(\Gamma,X)$ has the DLD2P but the Shilov boundary $\Gamma$ of the base algebra of $A(K, X)$ has an isolated point $t_0$. Since the set of strong boundary points $\Gamma_0$ is dense in $\Gamma$, this $t_0$ is also in $\Gamma_0$. Then by Lemma \ref{lem:auxiso} there is $\phi \in A$ such that $\phi(t_0) = 1=\|\phi\|$ and $\phi(t) = 0$ for all $t \in \Gamma \setminus \{t_0\}$. Since $\phi$ is in the base algebra $A$, there is $g\in A(\Gamma, X)$ such that $\phi = x^* \circ g$ for some $x^*\in X^*$. Then $g(t_0)\neq 0$ and let $h=\frac{\phi}{\|g(t_0)\|_X} g$. So $h$ is an element of $A(K, X)$. Since $\|h\|=\|h(t_0)\|=1$ and $\|h(t)\|=0$ for every point $t$ of $\Gamma$ other than $t_0$. So $\|h(t)\|=0$ for every limit point $t$ of $\Gamma$. Hence in view of Theorem \ref{th:deqakx}, we see that $h \in S_{A(K, X)}$ is not a $\Delta$-point, and this is a contradiction to our assumption that $A(K, X)$ has the DLD2P, which states that every function in $S_{A(K, X)}$ is a $\Delta$-point. Therefore (iii) $\implies$ (iv) is proved.
 \end{proof}

As we mentioned before, it is well-known that the nonexistence of isolated points on the Shilov boundary implies that the scalar-valued uniform algebra has the Daugavet property \cite{W2, Wo}. Here we have something even further that the Daugavet property, DD2P, and DLD2P are equivalent to each other.
\begin{Corollary}\label{lemma:equiDau}
		Let $K$ be a compact Hausdorff space and let $\Gamma$ be the Shilov boundary of a uniform algebra $A$. Then the following are equivalent:
	\begin{enumerate}[\rm(i)]
		\item $A$ has the Daugavet property. 
		\item $A$ has the DD2P.
		\item $A$ has the DLD2P.
		\item The Shilov boundary $\Gamma$ does not have isolated points.
	\end{enumerate}
\end{Corollary}

We conclude this article with showing when the space $A(K, X)$ has the convex-DLD2P. But here we take a slightly different approach. To use Lemma \ref{th:urysohn} we need to find a limit point of $\Gamma$ that is also a strong boundary point, but the existence of such point does not seem to be guaranteed in general. So here we use a different version of the Urysohn-type lemma for strong peak points. For a uniform algebra $A$ over a compact Hausdorff space $K$, an element $t_0 \in K$ is said to be a strong peak point if there exists a function $f \in A$ such that $|f(t_0)| = \|f\|_{\infty}$ and $\|f\|_{\infty} > \sup\{|f(t)| : t \in K \setminus U\}$ for every open neighborhood $U$ of $t_0$. Let us denote the set of all strong peak points for $A$ by $\rho(A)$.

\begin{Lemma}\cite[Lemma 3]{KL}\label{lem:urysohnsp}
	Let $\Omega$ be a Hausdorff space and $A$ be a subalgebra of $C_b(\Omega)$. Suppose that $U$ is an open subset of $\Omega$ and $t_0 \in U \cap \rho(A)$. Then given $0< \epsilon < 1$, there exists $\phi \in A$ such that $\|\phi\| = 1 = \phi(t_0)$, $\sup_{t \in K \setminus U} |\phi(t)| < \epsilon$, and for all $ t \in \Omega$,
	\[
	|\phi(t)| + (1- \epsilon)|1 - \phi(t)| \leq 1.
	\]
\end{Lemma}
The original statement of Lemma \ref{lem:urysohnsp} is given with $\|\phi\| = 1 = |\phi(t_0)|$. However, if we observe the proof of Lemma 3 in \cite{KL} carefully, it is easy to see that the statement without the absolute value also holds.  
  
\begin{Theorem}
	For a compact Hausdorff space $K$ and a uniformly convex Banach space $X$, suppose that a function space $A(K, X)$ over a base algebra $A$ satisfies the additional condition that $A \otimes X \subset A(K, X)$. Let $\Gamma'$ be the set of limit points in the Shilov boundary $\Gamma$ of the base algebra $A$ and let $\rho(A)$ the set of all strong peak points for $A$. If $\rho(A) \cap \Gamma' \neq \emptyset$, then the space $A(K, X)$ has the convex-DLD2P.
\end{Theorem}  

\begin{proof}
	By the isometry in Lemma~\ref{lem:AKXisom}, we may assume that $K=\Gamma$. We need to show that $S_{A(K, X)} \subset \overline{\text{conv}}\Delta$, where $\Delta$ is the set of all $\Delta$-points of $A(K, X)$. 
	
Fix $f \in S_{A(K, X)}$. Pick $t_0 \in \rho(A) \cap K'$ and let $\lambda = \frac{1 + \|f(t_0)\|_X}{2}$.
For $\epsilon > 0$, let $U$ be an open neighborhood containing $t_0$ such that $\|f(t) - f(t_0)\|_X < \epsilon$. By Lemma \ref{lem:urysohnsp}, there exists a function $\phi \in A$ such that $\phi(t_0) = \|\phi\|_\infty = 1$,  $\sup_{t \in K \setminus U}|\phi(t)| < \epsilon$ and
\begin{equation*}
|\phi(t)| + (1 - \epsilon)|1 - \phi(t)| \leq 1.
\end{equation*}
for every $t \in K$. Choose a norm-one vector $v_0 \in X$ and set 
\[x_0=\begin{cases} \frac{f(t_0)}{\|f(t_0)\|_X} &\mbox{if } f(t_0) \neq 0 \\
v_0  &\mbox{if } f(t_0)=0.
\end{cases} \]
Now define two functions
\begin{align*}
f_1(t) &= (1 - \epsilon)(1 - \phi(t))f(t) + \phi(t)x_0\\
f_2(t) &= (1 - \epsilon)(1 - \phi(t))f(t) - \phi(t)x_0, \ \ \ t\in K.
\end{align*}
Since $A \otimes X \subset A(K, X)$, we see that the functions $f_1, f_2 \in A(K, X)$. Notice that
\begin{align*}
\|f_1(t)\|_X &=\left\|(1 - \epsilon)(1 - \phi(t))f(t) + \phi(t) x_0\right\|_X \\&\leq (1 - \epsilon)|1 - \phi(t)| + |\phi(t)| \leq 1,  
\end{align*}
for every $t \in K$, in particular, $\|f_1(t_0)\|_X = 1$. Hence $\|f_1\| = \|f_1(t_0)\|_X = 1$. We can also show $\|f_2\| = \|f_2(t_0)\|_X = 1$ by the same argument. Thus by Theorem \ref{th:deqakx}, we see that $f_1, f_2 \in \Delta$. Let $g(t) = \lambda f_1(t) + (1 - \lambda) f_2(t)$. We consider two cases. 
\begin{enumerate}[\rm(i)]
	\item First, consider when $f(t_0) \neq 0$. Then $g  = (1 - \epsilon)(1 - \phi(t))f(t) + \phi(t)f(t_0)$. We see that 
	\begin{eqnarray*}
		\|g(t) - f(t)\|_X &=& \| (1 - \epsilon)(1 - \phi(t))f(t) + \phi(t)f(t_0) - f(t)\|_X\\
		&=& \| (1 - \epsilon)(1 - \phi(t))f(t) + \phi(t)f(t_0) - (1 -\epsilon)f(t) - \epsilon f(t)\|_X\\
		&=& \|(1-\epsilon)(- \phi(t))f(t) +  (1 - \epsilon)\phi(t)f(t_0) + \epsilon\phi(t)f(t_0) - \epsilon f(t)\|_X\\
		&=&\|(1 - \epsilon)\phi(t)(f(t_0) - f(t)) + \epsilon\phi(t)f(t_0) - \epsilon f(t)\|_X\\
		&\leq& (1 -\epsilon)|\phi(t)|\cdot \|f(t) - f(t_0)\|_X +  \epsilon|\phi(t)|\cdot\|f(t_0)\|_X+ \epsilon \|f(t)\|_X\\
		&\leq & (1 -\epsilon)|\phi(t)|\cdot \|f(t) - f(t_0)\|_X + 2\epsilon.
	\end{eqnarray*}
	
	For $t \in U$, we have $(1 -\epsilon)|\phi(t)|\cdot \|f(t) - f(t_0)\|_X \leq (1 -\epsilon)\epsilon< \epsilon$.  Since $|\phi(t)| < \epsilon$ for $t \in K \setminus U$, we have $(1 -\epsilon)|\phi(t)|\cdot \|f(t) - f(t_0)\|_X \leq 2(1 -\epsilon)\epsilon < 2 \epsilon$ by the triangle inequality. Therefore $\|g - f\| < 4 \epsilon$, and this implies that $f \in \overline{\text{conv}}\Delta$.
	
	\item Now consider when $f(t_0) = 0$. Then for all $t \in U$, $\|f(t)\|_X < \epsilon$. Notice that $\lambda = \frac{1}{2}$ and $g(t) = (1 - \epsilon)(1 - \phi(t))f(t)$. Hence 
	\begin{eqnarray*}
		\|g(t) - f(t)\|_X &=& \|(1 - \epsilon)(1 - \phi(t))f(t)  - (1 - \epsilon)f(t) - \epsilon f(t)\|_X\\
		&\leq& (1- \epsilon)|\phi(t)|\cdot \|f(t)\|_X + \epsilon \|f(t)\|_X \leq (1- \epsilon)|\phi(t)|\cdot \|f(t)\|_X + \epsilon.
	\end{eqnarray*}
	For $t \in U$, we have $(1- \epsilon)|\phi(t)|\cdot \|f(t)\|_X \leq (1 - \epsilon)\epsilon < \epsilon$. In view of Lemma \ref{lem:urysohnsp}, since $\sup_{t \in K \setminus U} |\phi(t)| < \epsilon$ for $t \in K \setminus U$, we have $(1- \epsilon)|\phi(t)|\cdot \|f(t)\|_X \leq (1- \epsilon)\epsilon < \epsilon$. Therefore $\|g - f\| < 2 \epsilon$, and this also implies that $f \in \overline{\text{conv}}\Delta$. 
\end{enumerate}
	  Since $f \in S_{A(K, X)}$ is arbitrary, we see that $S_{A(K,X)} \subset \overline{\text{conv}}(\Delta)$, and so $A(K, X)$ has the convex-DLD2P.        		
\end{proof}

\begin{corollary}\label{th:convexdld2p}
	Let $K$ be a compact Hausdorff space and $\Gamma'$ be the set of limit points in the Shilov boundary $\Gamma$ of a uniform algebra $A$. If $\rho(A) \cap \Gamma' \neq \emptyset$, then the uniform algebra $A$ has the convex-DLD2P.
\end{corollary}

\begin{Remark}
	The sufficient condition $\rho(A) \cap \Gamma' \neq \emptyset$ in Corollary~\ref{th:convexdld2p} does not guarantee the DLD2P. For example, let $K = \{\frac{1}{n} : n \in \mathbb{N}\} \cup \{0\}$ in the real line and let $A = C(K)$ be the space of continuous functions over $K$ equipped with the supremum norm. Notice that $K=\rho(A)$ is its Shilov boundary and $\rho(A) \cap \Gamma' = \{0\}$. So in view of Theorem~\ref{lemma:equiDau}  and Theorem \ref{th:convexdld2p}, the space $C(K)$ has the convex-DLD2P, while it does not have DLD2P. In fact, $C(K)$ is isometric to the space $c$ of convergent sequences. The same results in the space $c$ were shown in \cite[Corollary 5.4, Remark 5.5]{AHLP}.    
\end{Remark}

\end{document}